\documentclass[10pt,reqno]{amsart}

\usepackage{packages}

\title[{%
  Pointed Convex Sets and Rational Supporting Halfspaces
}]{%
  Pointed Closed Convex Sets are the Intersection\\
  of All Rational Supporting Closed Halfspaces
}
\author{Marcel K. de Carli Silva}
\address[Marcel K.\ de Carli Silva]{%
  Instituto de Matemática e Estatística, %
  Universidade de São Paulo
}
\email{mksilva@ime.usp.br}
\thanks{%
  Research of the first author was supported in part by FAPESP
  (Proc.~2013/03447-6), CNPq (Proc.~477203/2012-4), CNPq
  (Proc.~456792/2014-7), and CAPES%
}
\author{Levent Tunçel}
\address[Levent Tunçel]{%
  Department of Combinatorics and Optimization, University of Waterloo%
}
\email{ltuncel@uwaterloo.ca}
\thanks{%
  Research of the second author was supported in part by Discovery
  Grants from NSERC and by U.S.~Office of Naval Research under award
  numbers N00014-15-1-2171 and~N00014-18-1-2078.%
}

\date{February 8, 2018}

\begin{document}

\begin{abstract}
  We prove that every pointed closed convex set in~\(\Reals^n\) is the
  intersection of all the rational closed halfspaces that contain~it.
  This generalizes a previous result by the authors for compact convex
  sets.
\end{abstract}

\maketitle

A \emph{rational closed halfspace} is a subset of~\(\Reals^n\) of the
form \(\setst{x \in \Reals^n}{\iprod{a}{x} \leq \beta}\) for some \(a
\in \Rationals^n \setminus \set{0}\) and \(\beta \in \Rationals\).
In~\cite[Theorem~8]{CarliT18a} it is proved that every compact convex
set is the intersection of all the rational closed halfspaces that
contain~it.  In that paper, this result was a key step in generalizing
the polyhedral notion of total dual integrality
(see~\cite{Schrijver86a,Schrijver03a}) to more general convex sets.  A
natural question is whether the same is true for more general families
of convex sets.  Closedness is an obvious necessary condition for such
sets.  The statement is clearly false for arbitrary (in~fact, even
polyhedral) closed convex sets: if \(a \in \Reals^n\) has both
rational and irrational entries, then no rational closed halfspace
of~\(\Reals^n\) contains \(\setst{x \in \Reals^n}{\iprod{a}{x} \leq
  \beta}\), for~any \(\beta \in \Reals\).

In this short note, we generalize the result to pointed closed convex
sets, using elementary convex analysis.  We use standard notation
from~\cite{Rockafellar97a}, and we make extensive use of Minkowski set
operations.  The \emph{effective domain} of an extended real-valued
function \(f \colon \Reals^n \to \halfclosed{{-}\infty}{+\infty}\) is
\(\dom(f) \coloneqq \setst{x \in \Reals^n}{f(x) < +\infty}\).  Let \(C
\subseteq \Reals^n\) be a nonempty convex set.  The \emph{support
  function} of~\(C\) is \(\suppf{C}{a} \coloneqq \sup_{x \in C}
\iprod{a}{x} \in \halfclosed{{-}\infty}{+\infty}\) for each \(a \in
\Reals^n\), the \emph{barrier cone} of~\(C\) is \(B_C \coloneqq
\dom\paren*{\suppf{C}{\cdot}}\), the \emph{recession cone} of~\(C\) is
\(0^+C \coloneqq \setst{d \in \Reals^n}{\forall x \in C,\,x+\Reals_+d
  \subseteq C}\), and the \emph{polar} of~\(C\) is \(C^{\circ}
\coloneqq \setst{y \in \Reals^n}{\forall x\in C,\,\iprod{y}{x} \leq
  1}\).  The unit ball in~\(\Reals^n\) is \(\Ball \coloneqq \setst{x
  \in \Reals^n}{\norm{x} \leq 1}\).
\begin{lemma}
  \label{lem:barrier-cone-int}
  Let \(C \subseteq \Reals^n\) be a nonempty pointed closed convex
  set.  Then \(B_C\) has nonempty interior.
\end{lemma}
\begin{proof}
  Clearly \(B_C\) is a convex cone containing the origin.  Then
  \(B_C^{\circ} = 0^+C\) by~\cite[Corollary~14.2.1]{Rockafellar97a}
  whence \(\cl(B_C) = (0^+C)^{\circ}\).  Since \(0^+C\) is pointed,
  \(n = \dim\paren{(0^+C)^{\circ}} = \dim(\cl(B_C))\) whence
  \(\interior(B_C) = \interior(\cl(B_C))\) is nonempty.
\end{proof}

\begin{lemma}
  \label{lema:cone-wedge-int}
  Let \(x_0,d \in \Reals^n\) and \(\eps,\delta > 0\).  Then
  \(\conv\paren{\set{x_0} \cup (d+\eps\Ball)} \cap (x_0 + \delta\kern
  .5pt\Ball)\) has nonempty interior.
\end{lemma}
\begin{proof}
  Set \(\lambda \coloneqq
  \min\curly[\big]{\nicefrac{\delta}{(\norm{d-x_0}+\eps)}, 1 } \in
  \halfclosed{0}{1}\).  Then \(X \coloneqq (1-\lambda)x_0 +
  \lambda(d+\eps\Ball) \subseteq \conv\paren*{\set{x_0} \cup (d +
    \eps\Ball)}\) and the inclusion \(X \subseteq x_0 + \delta\kern
  .5pt \Ball\) is equivalent to \(\norm{\lambda(d+\eps u-x_0)} \leq
  \delta\) for each \(u \in \Ball\), which holds by the definition
  of~\(\lambda\).  Since \(\lambda > 0\) and \(\eps > 0\), it follows
  that \(\interior(X) \neq \emptyset\).
\end{proof}

\begin{theorem}
  \label{thm:main}
  Every pointed closed convex set is the intersection of all rational
  closed halfspaces that contain~it.
\end{theorem}
\begin{proof}
  \newcommand*{\Ctilde}{\widetilde{C}}
  Let \(X \subseteq \Reals^n\) be a pointed closed convex set.  We may
  assume that \(X \neq \emptyset\).  Clearly \(X\) is contained in the
  intersection of all rational closed halfspaces that contain~\(X\).
  Hence, it suffices to prove that for each \(\ytilde \in \Reals^n
  \setminus X\), there are \(a \in \Rationals^n\) and \(\beta \in
  \Rationals\) such that \(\iprod{a}{x} \leq \beta\) for each \(x \in
  X\) and \(\iprod{a}{\ytilde} > \beta\).  So let \(\ytilde \in
  \Reals^n \setminus X\).  Since~\(\Rationals\) is dense
  in~\(\Reals\), it suffices to prove that
  \begin{equation}
    \label{eq:1}
    \text{%
      there exists \(a \in \Rationals^n\) such that
      \(\suppf{X}{a} < \iprod{a}{\ytilde}\).
    }
  \end{equation}

  Let \(\ztilde \in X\) be the metric projection of~\(\ytilde\)
  in~\(X\), i.e., \(\set{\ztilde} = \argmin_{z \in X}
  \norm{z-\ytilde}\).  Set \(C \coloneqq X - \ztilde\) and \(\ybar
  \coloneqq \ytilde - \ztilde \neq 0\).  Then \(0 \in C\) and
  \begin{equation}
    \label{eq:3}
    \suppf{C}{a} \geq 0
    \qquad
    \text{for each \(a \in \Reals^n\), with equality if \(a = \ybar\)}.
  \end{equation}
  By \Cref{lem:barrier-cone-int}, there are \(d \in B_C\) and \(\eps >
  0\) such that the compact set \(d + \eps\Ball\) is a subset
  of~\(\interior(B_C)\).  Hence, \(\suppf{C}{\cdot}\) is Lipschitz
  continuous on \(d + \eps\Ball\) with Lipschitz constant at most \(L
  > 0\); see, e.g., \cite[Theorem~10.4]{Rockafellar97a}.  In
  particular,
  \begin{equation}
    \label{eq:4}
    \suppf{C}{d+\eps u} \leq \suppf{C}{d} + L\eps
    \qquad
    \forall u \in \Ball.
  \end{equation}
  Set
  \begin{gather*}
    \alpha \coloneqq
    \frac{\tfrac{1}{3}\norm{\ybar}^2}{\suppf{C}{d}+L\eps} > 0,
    \qquad
    \dbar \coloneqq \alpha d,
    \qquad
    \epsbar \coloneqq \alpha \eps > 0,
    \qquad
    \bar{\delta} \coloneqq \tfrac{1}{3}\norm{\ybar} > 0,
    \\[3pt]
    A \coloneqq \conv\paren*{\set{\ybar} \cup (\dbar + \epsbar\kern .8pt\Ball)}
    \cap (\ybar + \bar{\delta}\kern 1pt\Ball).
  \end{gather*}
  We claim that,
  \begin{equation}
    \label{eq:7}
    \suppf{C}{a} < \iprod{a}{\ybar}
    \qquad
    \forall a \in A.
  \end{equation}
  Let \(a \in A\).  So there exist \(\lambda \in [0,1]\) and \(\ubar
  \in \Ball\) such that \(a = (1-\lambda)\ybar +
  \lambda(\dbar+\epsbar\ubar)\).  Then
  \begin{equation}
    \label{eq:5}
    \suppf{C}{a}
    \leq
    (1-\lambda) \suppf{C}{\ybar} +
    \lambda\suppf{C}{\dbar+\epsbar\ubar}
    =
    \lambda\alpha\suppf{C}{d+\eps\ubar}
    \leq
    \tfrac{1}{3}\norm{\ybar}^2,
  \end{equation}
  where we used~\eqref{eq:3}, \eqref{eq:4}, and the fact that \(\lambda
  \leq 1\).  On the other hand, \(a = \ybar + \bar{\delta}\vbar\) for
  some \(\vbar \in \Ball\), so
  \begin{equation}
    \label{eq:6}
    \iprod{a}{\ybar}
    =
    \norm{\ybar}^2 + \bar{\delta}\iprod{\vbar}{\ybar}
    \geq
    \norm{\ybar}\paren*{\norm{\ybar} - \bar{\delta}}
    =
    \tfrac{2}{3}\norm{\ybar}^2.
  \end{equation}
  Combining~\eqref{eq:5} and~\eqref{eq:6} yields~\eqref{eq:7}.

  By adding \(\iprod{a}{\ztilde}\) to both sides of~\eqref{eq:7}, we
  find that \(\suppf{X}{a} < \iprod{a}{\ytilde}\) for each \(a \in
  A\).  By \Cref{lema:cone-wedge-int}, we~have \(\interior(A) \neq
  \emptyset\).  Hence, there exists a rational vector \(a\) in~\(A\).
  This proves~\eqref{eq:1} and the proof of the theorem is complete.
\end{proof}

The result is tight due to existence of closed halfspaces that are
contained in no rational closed halfspace, as~discussed~above.  Even
though \Cref{thm:main} does not directly yield a generalization of the
notion of total dual integrality in~\cite{CarliT18a} for pointed
closed convex sets (due to limitations of the Gomory-Chvátal closure),
the \namecref{thm:main} does provide a natural generalization of our
previous, foundational result for compact convex sets.

\end{document}